\documentclass[a4paper,12pt]{article}
\title{Superconvergence Using Pointwise Interpolation in
       Convection-Diffusion Problems}
\author{Sebastian Franz\thanks{
           Institut f\"ur Numerische Mathematik, Technische Universit\"at Dresden,
           01062 Dresden, Germany.
           \texttt{sebastian.franz@tu-dresden.de}\newline}}
\date{\today}

\usepackage{fancyhdr} 
\usepackage{pgfpages}

\usepackage{amsmath}
\usepackage{amssymb}
\usepackage{amsthm}

\usepackage{cite}

\usepackage{booktabs}

\usepackage{esint}

\usepackage{mathrsfs}
\usepackage{bbm}

\makeatletter
\let\my@saved@original@eqref\eqref 
\renewcommand*{\eqref}[1]{
  \begingroup
    \let\normalfont\relax
    \my@saved@original@eqref{#1}
  \endgroup
}
\makeatother
\makeatletter
\renewcommand*\env@matrix[1][r]{\hskip -\arraycolsep
  \let\@ifnextchar\new@ifnextchar
  \array{*\c@MaxMatrixCols #1}}
\makeatother

\DeclareMathOperator{\meas}{meas}
\newcommand{\e}{\mathrm{e}}

\newcommand{\laplace}{\Delta}
\newcommand{\grad}{\nabla}
\newcommand{\pt}{\partial}

\newcommand{\eps}{\varepsilon}
\newcommand{\norm}[2]{\|{#1}\|_{#2}}
\newcommand{\snorm}[2]{|{#1}|_{#2}}
\newcommand{\tnorm}[1]{\left|\!\!\;\left|\!\!\;\left| {#1}
                       \right|\!\!\;\right|\!\!\;\right|}
\newcommand{\enorm}[1]{\tnorm{#1}_\eps}
\newcommand{\half}{\frac{1}{2}}

\newcommand{\R}{\mathbb{R}}

\newcommand{\PS}{\mathcal{P}}
\newcommand{\QS}{\mathcal{Q}}

\theoremstyle{plain}
\newtheorem{thm}{Theorem}[section]
\newtheorem{lem}[thm]{Lemma}
\newtheorem{ass}[thm]{Assumption}

\newtheorem{cor}[thm]{Corollary}
\newtheorem{rem}[thm]{Remark}

\begin{document}
 \pagestyle{fancy}
  \maketitle
  \begin{abstract}
     Considering a singularly perturbed convection-diffusion problem,
     we present an analysis for a superconvergence result using pointwise
     interpolation of Gau\ss-Lobatto type for higher-order
     streamline diffusion FEM.
     We show a useful connection between two different types of interpolation,
     namely a vertex-edge-cell interpolant and a pointwise interpolant.
     Moreover, different postprocessing operators are analysed and applied
     to model problems.
  \end{abstract}



 \section{Introduction}
  Consider the convection-diffusion problem
  given by
  \begin{subequations}\label{eq:Lu}
  \begin{align}
   -\eps\laplace u-b u_x+cu&=f,\quad\mbox{in }\Omega=(0,1)^2\\
    u&=0,\hspace*{0.4cm}\mbox{on }\partial\Omega
  \end{align}
  \end{subequations}
  where $b\geq\beta>0$, $c+\half b_x\geq\gamma>0$ and $0<\eps\ll1$. Note that the condition on $c$ can
  always be fulfilled by a transformation $v=\exp(\kappa x)u$ for a suitably chosen $\kappa$.

  In \cite{FrL06} it was shown that for bilinear elements and a standard Galerkin method
  its solution $u^N$ fulfils on a piecewise uniform Shishkin mesh with $N$ mesh cells in each coordinate direction
  the estimates
  \begin{gather}\label{eq:supercloseness}
   \enorm{u-u^N}\leq C N^{-1}\ln N
   \quad\mbox{and}\quad
   \enorm{u^I-u^N}\leq C(N^{-1}\ln N)^2,
  \end{gather}
  where $u^I$ is the standard piecewise bilinear interpolant of $u$ and
  the energy-norm $\enorm{\cdot}$ is defined as
  \[
   \enorm{u}=(\eps\norm{\grad u}{0}^2+\gamma\norm{u}{0}^2)^{1/2}.
  \]
  Here and throughout the paper we denote by $\norm{\cdot}{0,D}$ the standard $L_2$-norm on $D\subset\Omega$
  and by $C$ a generic constant independent of $\eps$ and $N$.
  Whenever $D=\Omega$ we skip the explicit reference on the domain.
  
  A property like \eqref{eq:supercloseness} is called \emph{supercloseness}.
  It can be exploited as \emph{interpolantwise superconvergence}
  (for the naming convention see \cite{RST08})
  with a simple postprocessing routine \cite{FrL06}.
  The result is a better numerical solution $Pu^N$ fulfilling
  \[
   \enorm{u-Pu^N}\leq C(N^{-1}\ln N)^2.
  \]

  In \cite{FrLR06} a similar result was obtained for a streamline diffusion method \cite{HB79}
  under some restrictions on the stabilisation parameters. For higher order methods
  using $\QS_p$-elements with $p>1$
  so far only for the streamline diffusion method supercloseness results are known.
  In \cite{Fr11} it was proven that
  \[
   \enorm{\pi_p^N u-u^N}\leq C(N^{-1}\ln N)^{p+1/2}\ln N
  \]
  holds for the streamline diffusion solution $u^N$ in the case of $\QS_p$-elements
  on a suitable piecewise uniform Shishkin mesh and conditions on the stabilisation
  parameters. The interpolant $\pi_p^N$ is a so called vertex-edge-cell interpolant~\cite{Lin91,GR86}.

  In \cite{Fr12} the higher order case was investigated numerically and three unproven
  phenomena were shown. First, there seems to be a supercloseness result for pointwise
  interpolation w.r.t. Gau\ss-Lobatto points. Second, the supercloseness order is actually
  $p+1$ and not only $p+1/2$. And finally, these results do also hold for standard, unstabilised  Galerkin FEM.
  In the present paper we prove the first of these numerical results.

  The theoretical results presented in this paper require regularity of the exact solution $u$ and
  use solution decompositions given e.g. in \cite{KellSt05,KellSt07,LS01a}.
  We assume the necessary compatibility conditions and smoothness of the data to be fulfilled.

  The paper is organised as follows. In Section~\ref{sec:mesh} we introduce a class of layer-adapted
  meshes used in discretising the differential equation. Moreover, a solution decomposition
  exploited later in the analysis is presented. Section~\ref{sec:interpolation} contains
  the definition and analysis of some properties of two different interpolation operators.
  In Section~\ref{sec:analysis} we prove supercloseness properties
  of our two interpolation operators and in the following Section~\ref{sec:postprocessing}
  superconvergent numerical solutions are generated by postprocessing. Finally, Section~\ref{sec:numerics}
  contains a numerical example verifying the theoretical results.
  
\section{The Mesh and a Solution Decomposition}\label{sec:mesh}
  We define the underlying mesh as a member of the
  general class of S-type meshes \cite{RoosLins99}.
  Let the number $N\geq 4$ of mesh cells in each
  direction be divisible by 4 and
  a user-chosen positive parameter $\sigma>0$ be given. Assume
  \begin{gather}\label{eq:ass:eps}
   \eps\leq\frac{1}{(4\sigma\ln N)^2}.
  \end{gather}
  In practice, this assumption is no restriction as otherwise $N$ would be exponentially large
  compared with $\eps$. In the latter case the analysis could be done in a standard and much simpler way.
  We now define mesh transition parameters by
  \begin{gather}\label{eq:ass:trans}
    \lambda_x :=\frac{\sigma\eps}{\beta}\ln N \le \frac{1}{2}
   \quad\text{and}\quad
    \lambda_y :=\sigma\sqrt{\eps}\ln N          \le \frac{1}{4}.
  \end{gather}
  The domain $\Omega$ is dissected by a tensor product mesh
  according to
  \begin{align*}
   x_i&:=\begin{cases}
          \frac{\sigma\eps}{\beta_1}\phi\left(\frac{i}{N}\right),
                                          &i=0,\dots,N/2,\\
          1-2(1-\lambda_x)(1-\frac{i}{N}),&i=N/2,\dots,N,
         \end{cases}\\
   y_j&:=\begin{cases}
          \sigma\sqrt{\eps}\phi\left(\frac{2j}{N}\right),  &j=0,\dots,N/4,\\
          (1-2\lambda_y)(\frac{2j}{N}-1)+\half,            &j=N/4,\dots,3N/4,\\
          1-\sigma\sqrt{\eps}\phi\left(2-\frac{2j}{N}\right),&j=3N/4,\dots,N,
         \end{cases}
  \end{align*}
%
%
  and the final mesh $T^N$ is constructed by drawing lines parallel to the coordinate axes
  through these mesh points. 
  The function $\phi$ is a monotonically increasing, mesh-generating function satisfying
  $\phi(0)$=0 and $\phi(1/2)$=$\ln N$. Given an arbitrary function $\phi$ fulfilling
  these conditions, an S-type mesh is defined
  and the domain $\Omega$ is divided into the subdomains $\Omega_{11}$,
  $\Omega_{12}$, $\Omega_{21}$, and $\Omega_{22}$ as shown in
  Figure~\ref{fig:mesh}.
  \begin{figure}[tb]
   \begin{center}
     \begin{minipage}[c]{0.35\textwidth}
      \vspace*{0cm}
      \setlength{\unitlength}{0.42pt}
      \begin{picture}(256,256)
       \put(  0,  0){\line( 0, 1){256}}
       \put(  4,  0){\line( 0, 1){256}}
       \put(  8,  0){\line( 0, 1){256}}
       \put( 12,  0){\line( 0, 1){256}}
       \put( 76,  0){\line( 0, 1){256}}
       \put(136,  0){\line( 0, 1){256}}
       \put(196,  0){\line( 0, 1){256}}
       \put(256,  0){\line( 0, 1){256}}

       \put(  0,  0){\line( 1, 0){256}}
       \put(  0, 19){\line( 1, 0){256}}
       \put(  0, 83){\line( 1, 0){256}}
       \put(  0,128){\line( 1, 0){256}}
       \put(  0,173){\line( 1, 0){256}}
       \put(  0,237){\line( 1, 0){256}}
       \put(  0,256){\line( 1, 0){256}}

       \linethickness{1pt}
       \put( 16,  0){\line( 0, 1){256}}
       \put(  0, 37){\line( 1, 0){256}}
       \put(  0,217){\line( 1, 0){256}}

       \put(-5,-5){\makebox(0,0)[t]{$0$}}
       \put(16,-5){\makebox(0,0)[t]{$\lambda_x$}}
       \put(-5,38){\makebox(0,0)[r]{$\lambda_y$}}
       \put(-5,218){\makebox(0,0)[r]{$1-\lambda_y$}}
       \put(256,-5){\makebox(0,0)[t]{$1$}}
       \put(-5,256){\makebox(0,0)[r]{$1$}}

      \end{picture}
     \end{minipage}
     \begin{minipage}[c]{0.5\textwidth}
      \begin{align*}
       \Omega_{11}&:=[\lambda_x,1]\times [\lambda_y,1-\lambda_y],\\
       \Omega_{12}&:=[0,\lambda_x]\times [\lambda_y,1-\lambda_y],\\
       \Omega_{21}&:=[\lambda_x,1]\times \big([0,\lambda_y]\cup[1-\lambda_y,1]\big),\\
       \Omega_{22}&:=[0,\lambda_x]\times \big([0,\lambda_y]\cup[1-\lambda_y,1]\big)\\
      \end{align*}
     \end{minipage}
   \end{center}
   \caption{Shishkin mesh $T^8$ of $\Omega$, the bold lines indicate the boundaries
   of the subdomains.\label{fig:mesh}}
  \end{figure}
  Related to the mesh-generating function $\phi$, we define the mesh-characterising function
  $
    \psi=e^{-\phi}.
  $
  Its derivative yields information on
  the approximation quality of the mesh, usually expressed in terms of $\max|\psi'|:=\max\{|\psi'(t)|,\,t\in[0,1/2]\}$.
  Several examples can be found in \cite{RoosLins99}. In this paper we refer to two of them
  repeatedly. Those are
  the piecewise uniform Shishkin-mesh with $\phi(t)=2t\ln N$ and $\max|\psi'|\leq 2\ln N$, 
  and the Bakhvalov--Shishkin mesh with $\phi(t)=-\ln(1-2t(1-N^{-1}))$ and $\max|\psi'|\leq 2$. 
  Besides above properties, we assume the function $\phi$ also to fulfil
  \[
   \max_{t\in[0,1/2]}\phi'(t)\leq CN
   \quad\mbox{and}\quad
   \min_{i=1,\dots,N/2}\phi\left(\frac{i}{N}\right)-\phi\left(\frac{i-1}{N}\right)\geq CN^{-1}.
  \]
  With the help of the first property, the mesh sizes in the fine-mesh region
  can be estimated, see \eqref{eq:hi}. The second one is used in estimating
  the interpolation error \cite[Theorem 12]{FrM10} when applying an inverse inequality.
  Both properties are fulfilled for the Shishkin and the Bakhvalov-Shishkin mesh,
  and many others.

  We denote by $\tau_{ij}=[x_{i-1},x_i]\times[y_{j-1},y_j]$ a specific element
  and by $\tau$ a generic mesh rectangle. Note that the mesh cells are assumed
  to be closed. Let $h_i:=x_i-x_{i-1}$, $k_j:=y_j-y_{j-1}$ be the dimensions of $\tau_{ij}$ and
  \[
     h:=\max_{i=1,\dots,N/2} h_i,\quad
     k:=\max_{j=1,\dots,N/4} k_j\quad\mbox{and}\quad
      h_{min}:=\min_{i=1,\dots,N/2} h_i.
  \]
  Note that it holds, \cite[eq. (3.5)]{RoosLins99}
  \begin{gather}\label{eq:hi}
   h_i\leq C\eps N^{-1}\max|\psi'|e^{\beta x/(\sigma\eps)},\quad i=1,\dots,N/2,\,\,x\in [x_{i-1},x_i],
  \end{gather}
  and similarly for $k_j$.

  Let for a fixed polynomial degree $p\geq 2$ the finite element space be given by
  \[
   V^N=\left\{v\in H_0^1(\Omega): v|_\tau\in \QS_p(\tau)\,\forall\tau\in T^N \right\}.
  \]
%
  \begin{ass}\label{ass:dec}
   The solution $u$ of~\eqref{eq:Lu} can be decomposed as
   \begin{gather*}
    u =v+w_1+w_2+w_{12},
   \end{gather*}
   where for all $x,y\in[0,1]$ and $0\le i+j\le p+1$ the
   pointwise estimates
   \begin{equation}\label{eq:dec:C0}
    \left.
      \begin{aligned}
        \left|\frac{\partial^{i+j} v}{\partial x^i \partial y^j}(x,y)\right|
             &\le C,
        \quad
        \left|\frac{\partial^{i+j} w_1}{\partial x^i \partial y^j}(x,y)\right|
              \le C\eps_{}^{-i}\e_{}^{-\beta x/\eps},\\[0.2cm]
        \left|\frac{\partial^{i+j} w_2}{\partial x^i \partial y^j}(x,y)\right|
             &\le C\eps_{}^{-j/2}
                \left(\e_{}^{-y/\sqrt\eps}+\e_{}^{-(1-y)/\sqrt\eps} \right), \\[0.2cm]
        \left|\frac{\partial^{i+j} w_{12}}{\partial x^i \partial y^j}(x,y)\right|
             &\le C\eps^{-(i+j/2)} \e^{-\beta x/\eps}
                    \left(\e^{-y/\sqrt\eps}+\e^{-(1-y)/\sqrt\eps}\right)
      \end{aligned}
    \right\}
   \end{equation}

   and for $i+j=p+2$ the $L_2$-norm bounds
   \begin{equation}\label{eq:dec:L2}
    \left.
      \begin{aligned}
        \norm{\partial_x^i\partial_y^jv  }{0,\Omega}&\leq C,
        \hspace*{2.05cm}
        \norm{\partial_x^i\partial_y^jw_1}{0,\Omega} \leq C\eps^{-i+1/2},\\[0.2cm]
        \norm{\partial_x^i\partial_y^jw_2}{0,\Omega}&\leq C\eps^{-j/2+1/4},
        \quad
        \norm{\partial_x^i\partial_y^jw_{12}}{0,\Omega}\leq C\eps^{-(i+j/2)+3/4}
      \end{aligned}
    \right\}
   \end{equation}
   hold.
   Here $w_1$ covers the exponential boundary layer, $w_2$ the characteristic
   boundary layers, $w_{12}$ the corner layers, and $v$ is the regular part.
  \end{ass}
%
  \begin{rem}\label{rem:plaus}
    In \cite{KellSt05,KellSt07} Kellogg and Stynes proved the validity of
    Assumption~\ref{ass:dec} for constant functions $b, c$ under certain
    compatibility and smoothness conditions on $f$.
  \end{rem}

\section{Interpolation}\label{sec:interpolation}
  We define two different interpolation operators.
  The first one \cite{GR86,Lin91} is the vertex-edge-cell interpolation
  operator $\hat\pi_p:C(\hat\tau)\to{Q}_p(\hat\tau)$,
  defined locally on the reference element $\hat\tau:=[-1,1]^2$ by
  \begin{subequations}\label{eq:Lin:def}
  \begin{alignat}{2}
   (\hat\pi_p \hat{v}-\hat{v})(\hat{a}_i)&=0,\,\quad i=1,\dots,4,
   &&\label{eq:general:def_inter1}\\
   \int_{\hat{e}_i}(\hat\pi_p\hat{v}-\hat{v})\hat{q} &= 0,
   \quad i=1,\dots,4,\quad
   &&\hat{q}\in {P}_{p-2}(\hat{e}_i),\label{eq:general:def_inter2}\\
   \iint_{\hat{\tau}} (\hat\pi_p\hat{v}-\hat{v})\hat{q} &= 0,
   &&\hat{q}\in {Q}_{p-2}(\hat{\tau}),\label{eq:general:def_inter3}
  \end{alignat}
  where $\hat a_i$ are the vertices and $\hat e_i$ the edges of $\hat\tau$.
  Using the bijective reference mapping $F_\tau:\hat{\tau}\to\tau$,
  this operator can be extended to the global interpolation operator
  $\pi_p^N:C(\overline{\Omega})\to V^N$ by
  \begin{equation}\label{eq:def_inter_int}
   (\pi_p^N v)|_\tau:= (\hat\pi_p(v\circ F_\tau)) \circ F_\tau^{-1},
   \quad\forall\tau\in T^N,\,v\in
   C(\overline{\Omega}).
  \end{equation}
  \end{subequations}

  The second interpolation operator is of Lagrange-type.
  Let $-1=t_0<t_1<\dots<t_p=1$, be the zeros of
  \begin{subequations}\label{eq:GL:def}
  \begin{gather}\label{eq:GL:nodes}
    (1-t^2)L_{p}'(t)=0,\quad t\in[-1,1],
  \end{gather}
  where $L_p$ is the Legendre polynomial of degree $p$.
  These points are also used in the Gau\ss-Lobatto
  quadrature rule of approximation order $2p-1$. Therefore, we refer
  to them as Gau\ss-Lobatto points.
  In literature they are also named Jacobi points \cite{Li04} as they are also the zeros
  of the orthogonal Jacobi-polynomials $P_p^{(1,1)}$ of order $p$.

  The operator $\hat I_p:C(\hat\tau)\to{Q}_p(\hat\tau)$ is then defined on the reference element $\hat\tau$
  by point evaluations
  \begin{gather}\label{eq:GL:def_loc}
    (\hat I_p\hat v)(t_i,t_j)=\hat v(t_i,t_j),\quad i,\,j=0,\dots,p.
  \end{gather}
  \end{subequations}
  With an extension like \eqref{eq:def_inter_int} we obtain the global interpolation operator $I^N_p:C(\overline{\Omega})\to V^N$.
  The interpolation error for both operators can be
  bounded according to \cite{FrM10,FrM10_1} using Assumption~\ref{ass:dec}.
%
  \begin{thm}\label{thm:interpolation}
   Let $\sigma\geq p+1$. Then it holds for the solution $u$ of \eqref{eq:Lu}
   \[
    \enorm{u-I^N_pu}\leq C(h+k+N^{-1}\max |\psi'|)^p
    \quad\mbox{and}\quad
    \enorm{u-\pi_p^Nu}\leq C(h+k+N^{-1}\max |\psi'|)^p.
   \]
  \end{thm}
  
  \begin{rem}
    Note that
    \[
     \max|\psi'|\leq \begin{cases}
                      2 \ln N,&\mbox{Shishkin mesh},\\
                      2, &\mbox{Bakhvalov--Shishkin mesh},
                     \end{cases}
    \]
    which shows the improvement of the bounds using graded meshes near the boundaries.
  \end{rem}

  \begin{lem}\label{lem:connection}
   Let $\hat\pi_p$ and $\hat I_p$ be the local interpolation operators into $\QS_p(\hat \tau)$
   on the reference element $\hat\tau$
   defined in \eqref{eq:Lin:def} and \eqref{eq:GL:def}, respectively.
   Furthermore, let $\hat \pi_{p+1}$ be defined similarly to \eqref{eq:Lin:def} interpolating into
   the local space $\QS_{p+1}(\hat\tau)$. Then it holds
   \begin{gather}\label{eq:equivalence:1}
    \hat\pi_p=\hat I_p\hat\pi_{p+1}.
   \end{gather}
   Moreover, if $\hat I_{p+1}^*$ is a Lagrange-interpolant into
   $\QS_{p+1}(\hat\tau)$ using the interpolation nodes $(t_i^*,t_j^*)$ where
   $\{t_i^*\},\, i=0,\dots,p+1$ consists of the $p+1$ Gau\ss-Lobatto nodes
   $\{t_i\},\,i=0,\dots,p$\, from \eqref{eq:GL:nodes}
   plus one arbitrary node $t_{p+1}^*\in(-1,1)$, then it follows
   \begin{gather}\label{eq:equivalence:2}
    \hat I_p=\hat \pi_p\hat I_{p+1}^*.
   \end{gather}
  \end{lem}
  \begin{proof}
   We extend an idea given in \cite{HX08}. For any function $v\in C(\hat\tau)$ holds
   \[
    \hat\pi_pv\in \QS_p(\hat\tau)
    \quad\mbox{and}\quad
    \hat I_pv\in \QS_p(\hat\tau).
   \]
   To prove the equivalence \eqref{eq:equivalence:1} we only have to show that
   $\hat I_p\hat\pi_{p+1}$ shares the same degrees of freedom as $\hat\pi_p$.
   The definitions \eqref{eq:Lin:def} and \eqref{eq:GL:def} imply
   \begin{subequations}\label{eq:conection:proof}
   \begin{gather}
    (\hat\pi_p v)(\hat a_i)=(\hat I_p\hat\pi_{p+1}v)(\hat a_i)=v(\hat a_i),\quad i=1,\dots,4,
   \end{gather}
   where $\hat a_i$ are the vertices of $\hat\tau$.
   Furthermore, it holds for $\hat q\in \PS_{p-2}(\hat e_i)$
   \begin{align*}
    \int_{\hat e_i}(v-\hat I_p\hat\pi_{p+1}v)\hat q
     &= \int_{\hat e_i}(v-\hat I_p\hat\pi_{p+1}v-(v-\hat\pi_{p+1}v))\hat q
      = \int_{\hat e_i}(\hat\pi_{p+1}v-\hat I_p(\hat\pi_{p+1}v))\hat q,
   \end{align*}
   where $\hat e_i$ is any of the four edges of $\hat\tau$.
   Now, the final integrand is a polynomial of order $2p-1$.
   Thus the integral can be rewritten using a quadrature rule
   that is exact for polynomials of order $2p-1$. We use the Gau\ss-Lobatto rule and obtain
   \begin{gather}
    \int_{\hat e_i}(v-\hat I_p\hat\pi_{p+1}v)\hat q
      =\sum_{j=0}^p w_j((\hat\pi_{p+1}v-\hat I_p(\hat\pi_{p+1}v))\hat q)|_{\hat e_i}(t_j)=0,\quad i=1,\dots,4,
   \end{gather}
   where $\{w_j\}$ are the weights of the quadrature rule.
   The last equality comes from \eqref{eq:GL:def_loc}.
   In a similar fashion it follows for $\hat q\in \QS_{p-2}(\hat\tau)$
   \begin{gather}
    \iint_{\hat \tau}(v-\hat I_p\hat\pi_{p+1}v)\hat q
      =\sum_{i,j=0}^p w_{i,j}((\hat\pi_{p+1}v-\hat I_p(\hat\pi_{p+1}v))\hat q)(t_i,t_j)=0,
   \end{gather}
   due to the integrand being a polynomial in $\QS_{2p-1}(\hat \tau)$,
   the Gau\ss-Lobatto rule on the rectangle $\hat\tau$ and again \eqref{eq:GL:def_loc}.
   Comparing \eqref{eq:conection:proof} to \eqref{eq:Lin:def}
   one concludes \eqref{eq:equivalence:1}.

   The second equivalence \eqref{eq:equivalence:2} can be concluded easily by the first one:
   \[
    \hat \pi_p \hat I_{p+1}^*
     =\hat I_p \hat\pi_{p+1} \hat I_{p+1}^*
     =\hat I_p \hat I_{p+1}^*
     =\hat I_p,
   \]
   where we use the property of $\hat I_{p+1}^*$ and $\hat\pi_{p+1}$ being projections into $\QS_{p+1}(\hat \tau)$
   in the second step. The last equality holds because $\hat I_p$ uses a subset of
   interpolation nodes of $\hat I_{p+1}^*$ in its definition.
   \end{subequations}
  \end{proof}

\section{Supercloseness Analysis}\label{sec:analysis}
  Let us now come to the numerical method.
  We define the Galerkin bilinear form by
  \[
    a_{Gal}(v,w) := \eps(\grad v,\grad w) + (c v-b v_x, w)
  \]
  and a stabilisation bilinear-form of the streamline-diffusion method~\cite{HB79} by
  \[
   a_{stab}(v,w):=\sum_{\tau\in T^N}\delta_\tau(\eps\laplace v+b v_x-c v,b w_x)_\tau,
  \]
  where the parameters $\delta_\tau\geq 0$ are user chosen and influence both stability
  and convergence. They are taken constant in each sub-domain of $\Omega$,
  i.e. $\delta_\tau=\delta_{ij}$ for any $\tau\subset\Omega_{ij}$.
  The specific bounds of $\delta_{ij}$ will be defined below, see Theorem~\ref{thm:Lin:superclose}.

  The streamline-diffusion bilinear-form is then defined as
  \[
   a_{SD}(v,w):=a_{Gal}(v,w)+a_{stab}(v,w)
  \]
  and the streamline-diffusion formulation of~\eqref{eq:Lu} is given by\bigskip

  Find ${u}^N\in V^N$ such that
  \begin{equation}\label{eq:SD_form}
    a_{SD}({u}^N, v^N) = f_{SD}(v^N),\qquad \forall v^N\in V^N
  \end{equation}
  where
  \[
   f_{SD}(v):=(f,v)-\sum_{\tau\in T^N}\delta_\tau(f,bv_x)_\tau.
  \]
%
  \begin{thm}\label{thm:Lin:superclose}
%
   For $\QS_p$-elements, $\sigma\geq p+1$ and under the restrictions on the stabilisation parameters
   \[
    \delta_{11}=C N^{-1},\quad
    \delta_{21}\leq C\max\{1,\eps^{-1/2}(N^{-1}\max |\psi'|)\}(N^{-1}\max |\psi'|)^{2},\quad
    \delta_{12}=\delta_{22}=0,
   \]
   it holds the estimate
   \[
    \enorm{\pi_p^N u-u^N}\leq C (h+k+N^{-1}\max |\psi'|)^{p+1/2}(\max|\psi'|\ln N)^{1/2}.
   \]
  \end{thm}
  \begin{proof}
    In \cite{Fr11} this result is given for the standard Shishkin-mesh. 
    Together with techniques for S-type meshes, see e.g.~\cite{RoosLins99,FrM10},
    its proof can be adapted directly and the desired bound follows.
    Note that the additional logarithmic factor is caused by the estimation
    of the convective term inside the characteristic layers.
  \end{proof}

  \begin{rem}
   If we replace above bound on $\delta_{21}$ by
   \[
    \delta_{21}= C(N^{-1}\max |\psi'|)^{2},
   \]
   the same argumentation that yields Theorem~\ref{thm:Lin:superclose}
   gives the slightly sharper result
   \[
     \enorm{\pi_p^N u-u^N}\leq C (h+k+N^{-1}\max|\psi'|)^{p+1/2}(\ln N)^{1/2}.
   \]
   Of course, if $\max|\psi'|\leq C$ (like in the case of a Bakhvalov--Shishkin mesh)
   this is the same bound.
  \end{rem}

  \begin{cor}\label{cor:convergence}
    Combining Theorems~\ref{thm:interpolation} and \ref{thm:Lin:superclose}
    yields the convergence result
    \[
     \enorm{u-u^N}\leq C(h+k+N^{-1}\max|\psi'|)^p.
    \]
  \end{cor}

  To analyse the supercloseness behaviour of the Gau\ss-Lobatto
  interpolation operator, consider
  \begin{gather}\label{eq:GL:connection}
   I^N_pu=\pi^N_pu+R^Nu+(u-\pi^N_{p+1}u),
  \end{gather}
  that is a consequence of Lemma~\ref{lem:connection}, where
  \[
    R^Nu:=I_p^N(\pi_{p+1}^Nu-u)-(\pi_{p+1}^Nu-u).
  \]
  Now \eqref{eq:GL:connection} implies
  \begin{gather}\label{eq:GL:start}
   \enorm{I_p^Nu-u^N}
    \leq \enorm{\pi_p^Nu-u^N}
        +\enorm{R^Nu)}
        +\enorm{\pi_{p+1}^Nu-u}.
  \end{gather}
  Its first term can be estimated by the supercloseness result of Theorem~\ref{thm:Lin:superclose}
  and its last term by the interpolation error result of Theorem~\ref{thm:interpolation}
  adapted to the case of elements of order $p+1$.
  Thus, we only have to estimate the energy norm of $R^Nu$.
  We start with some basic estimates for $R^Nu$.
%
  \begin{lem}\label{lem:Ru}
   For any $w\in C(\tau_{ij})$ holds the stability estimate
   \begin{subequations}
   \begin{gather}\label{eq:stab:1}
     \norm{R^Nw}{L_\infty(\tau_{ij})}\leq C\norm{w}{L_\infty(\tau_{ij})}.
   \end{gather}
%
   For any $w\in H^{p+2}(\Omega)$, $1\leq t\leq p$ and $\tau_{ij}\in T^N$ we have
   the anisotropic error estimates
   \begin{align}
    \norm{R^Nw}{0,\tau_{ij}}
     &\leq C\sum_{r=0}^{t+2}\norm{h_i^{t+2-r}k_j^r
                                  \pt_x^{t+2-r}\pt_y^r w}{0,\tau_{ij}},\label{eq:Ru:0}\\
    \norm{(R^Nw)_x}{0,\tau_{ij}}
     &\leq C\sum_{r=0}^{t+1}\norm{h_i^{t+1-r}k_j^r
                                  \pt_x^{t+2-r}\pt_y^r w}{0,\tau_{ij}},\label{eq:Ru:x}
   \end{align}
   and analogously for $\norm{(R^Nw)_y}{0,\tau_{ij}}$.
   \end{subequations}
  \end{lem}
  \begin{proof}
   The stability estimate \eqref{eq:stab:1} is a direct consequence of the stability of the interpolation
   operators $I^N_p$ and $\pi_{p+1}^N$ in $L_\infty$. Their stability holds because all degrees of freedom
   are point-evaluations or integrals.

   For \eqref{eq:Ru:0} we use anisotropic error estimates \cite{Apel99,FrM10,Matt08} to obtain
   \begin{align*}
    \norm{R^Nw}{0,\tau_{ij}}
     &\leq C\bigg[
              \norm{h_i(\pi_{p+1}^Nw-w)_x}{0,\tau_{ij}}+
              \norm{k_j(\pi_{p+1}^Nw-w)_y}{0,\tau_{ij}}
            \bigg]\\
     &\leq C\bigg[
              \sum_{r=0}^{t+1}\norm{h_i^{t+2-r}k_j^r
                                    \pt_x^{t+2-r}\pt_y^r w}{0,\tau_{ij}}
             +\sum_{r=0}^{t+1}\norm{h_i^{t+1-r}k_j^{1+r}
                                    \pt_x^{t+1-r}\pt_y^{r+1} w}{0,\tau_{ij}}
            \bigg],
   \end{align*}
   which gives \eqref{eq:Ru:0}.
   For \eqref{eq:Ru:x} we need additionally anisotropic estimates for the second order
   derivatives, see again \cite{Apel99}
   \begin{align*}
    \norm{(R^Nw)_x}{0,\tau_{ij}}
     &\leq C\bigg[
              \norm{h_i(\pi_{p+1}^Nw-w)_{xx}}{0,\tau_{ij}}+
              \norm{k_j(\pi_{p+1}^Nw-w)_{xy}}{0,\tau_{ij}}
            \bigg]\\
     &\leq C\bigg[
              \sum_{r=0}^{t}\norm{h_i^{1+t-r}k_j^r
                                  \pt_x^{t+2-r}\pt_y^r w}{0,\tau_{ij}}+
              \sum_{r=0}^{t}\norm{h_i^{t-r}k_j^{1+r}
                                  \pt_x^{t+1-r}\pt_y^{r+1} w}{0,\tau_{ij}}
            \bigg]
   \end{align*}
   which gives \eqref{eq:Ru:x}.
  \end{proof}

  \begin{thm}\label{thm:Ru}
   Let us assume $\sigma\geq p+1$ and $N^{-1}(\max|\psi'|)^2\leq C$.
   For $R^N$ defined above holds in the case of $k\leq C N^{-1/4}$
   \begin{subequations}\label{eq:Ru:energy}
   \begin{gather}
    \enorm{R^Nu}\leq C(N^{-(\sigma-1/2)}(1+h_{mesh})+(h+k+N^{-1}\max|\psi'|)^{p+1})
   \end{gather}
   and otherwise
   \begin{gather}
    \enorm{R^Nu}\leq C(N^{-(\sigma-2/3)}(1+h_{mesh})+(h+k+N^{-1}\max|\psi'|)^{p+1}),
   \end{gather}
   \end{subequations}
   where $h_{mesh}:=\eps N^{-1}(\ln N)^{1/2}/h_{min}$.
  \end{thm}
  \begin{rem}
   For the mesh specific value $h_{mesh}$ holds on a Shishkin mesh
   \[
    C_1(\ln N)^{-1/2}\leq h_{mesh}\leq C_2(\ln N)^{-1/2},
   \]
   on a Bakhvalov--Shishkin mesh
   \[
    C_1(\ln N)^{1/2}\leq h_{mesh}\leq C_2(\ln N)^{1/2},
   \]
   and on a general S-type mesh
   \[
    C_1(\ln N)^{-1/2}\leq h_{mesh}\leq C_2(\ln N)^{1/2}.
   \]
  \end{rem}

  \begin{proof}[Proof of Theorem ~\ref{thm:Ru}]
    Let us start with the $L_2$-norm estimate. 
    We use the solution decomposition of Assumption~\ref{ass:dec} and start with the regular part $v$.
    By \eqref{eq:Ru:0} with $t=p-1\geq1$ we obtain
    \begin{subequations}\label{eq:Ru:L2:proof}
      \begin{gather}
        \norm{R^Nv}{0,\Omega}\leq C(h+k+N^{-1})^{p+1}.
      \end{gather}
      Estimate \eqref{eq:Ru:0} can also be used to bound $w_1$ in $\Omega_{12}\cup\Omega_{22}$, where
      $h_i$ can be estimated by \eqref{eq:hi}:
      \begin{align*}
      \norm{R^Nw_1}{0,\Omega_{12}\cup\Omega_{22}}^2
       &\leq C \sum_{\tau_{ij}\subset\Omega_{12}\cup\Omega_{22}}
               \sum_{r=0}^{p+1}\norm{h_i^{p+1-r}k_j^r
                               \pt_x^{p+1-r}\pt_y^r w_1}{0,\tau_{ij}}^2\notag\\
       &\leq C \sum_{r=0}^{p+1}\norm{(\eps N^{-1}\max|\psi'|e^{\beta x/(\sigma\eps)})^{p+1-r}(k+N^{-1})^r
                               \eps^{-(p+1-r)} e^{-\beta x/\eps}}{0,\Omega_{12}\cup\Omega_{22}}^2\notag\\
       &\leq C (k+N^{-1}\max|\psi'|)^{2(p+1)}
               \sum_{r=0}^{p+1}\norm{e^{\frac{\beta x}{\eps}(\frac{p+1-r}{\sigma}-1)}}{0,\Omega_{12}\cup\Omega_{22}}^2.
      \end{align*}
      Now $(p+1)/\sigma-1\leq 0$ gives a non-positive exponent of the exponential function.
      Together with the rather crude bound $\meas(\Omega_{12}\cup\Omega_{22})\leq C$ follows
      \begin{align}
      \norm{R^Nw_1}{0,\Omega_{12}\cup\Omega_{22}}^2
       &\leq C (k+N^{-1}\max|\psi'|)^{2(p+1)}
               \sum_{r=0}^{p+1}\norm{e^{\frac{\beta x}{\eps}(\frac{p+1-r}{\sigma}-1)}}{0,\Omega_{12}\cup\Omega_{22}}^2\notag\\
       &\leq C (k+N^{-1}\max |\psi'|)^{2(p+1)}.
      \end{align}
      In $\Omega_{11}\cup\Omega_{21}$ we use the stability \eqref{eq:stab:1} to obtain
      \begin{align}
        \norm{R^Nw_1}{0,\Omega_{11}\cup\Omega_{21}}
         &\leq C\norm{w_1}{L_\infty(\Omega_{11}\cup\Omega_{21})}
          \leq C N^{-\sigma}.
      \end{align}
      The other two layer terms can be estimated similarly and combining these results
      proves the $L_2$-estimates.
     \end{subequations}
     For the $H^1$-component we use \eqref{eq:Ru:x} with $t=p$ and its counterpart for the $y$-derivative for the regular
     solution component $v$ to obtain
     \begin{subequations}\label{eq:Ru:x:proof}
       \begin{gather}
          \norm{\grad R^Nv}{0,\Omega}
           \leq C (h+k+N^{-1})^{p+1}.
       \end{gather}
       Similarly, using the type of analysis as above, we show
       \begin{align}
        \norm{\grad R^Nw_1}{0,\Omega_{12}\cup\Omega_{22}}
         &\leq C\eps^{-1/2}(k+N^{-1}\max |\psi'|)^{p+1},\\
        \norm{\grad R^Nw_2}{0,\Omega_{21}\cup\Omega_{22}}
         &\leq C\eps^{-1/4}(h+N^{-1}\max |\psi'|)^{p+1},\\
        \norm{\grad R^Nw_{12}}{0,\Omega_{22}}
         &\leq C\eps^{-1/4}(N^{-1}\max |\psi'|)^{p+1}.
       \end{align}
       On the other domains we use the decay of the layer terms.
       We show the analysis exemplary for the three terms yielding the
       largest bounds and invoking the most assumptions.
       Let us start with $(R^Nw_{12})_x$ in $\Omega_{12}$.
       By \eqref{eq:Ru:x} with $t=1$ we obtain
       \begin{align*}
         \norm{(R^Nw_{12})_x}{0,\Omega_{12}}
          &\leq C \sum_{r=0}^2(\eps N^{-1}\max|\psi'|)^{2-r}N^{-r}\eps^{-(3-r)}\eps^{-r/2}
                             \norm{e^{\frac{\beta x}{\eps}(\frac{2-r}{\sigma}-1)}(e^{-\frac{y}{\eps^{1/2}}}+e^{-\frac{1-y}{\eps^{1/2}}})}
                                  {0,\Omega_{12}}\notag\\
          &\leq C \eps^{-1/2}N^{-\sigma}[(N^{-1}\max|\psi'|)^2\eps^{-3/4}].
       \end{align*}
       On the other hand, a triangle and an inverse inequality give
       \[
        \norm{(R^Nw_{12})_x}{0,\Omega_{12}}
         \leq C \left[h_{min}^{-1}(\norm{I_p^N(\pi_{p+1}^Nw_{12}-w_{12})}{0,\Omega_{12}}+
                             \norm{\pi_{p+1}^Nw_{12}}{0,\Omega_{12}})
                +\norm{(w_{12})_x}{0,\Omega_{12}}\right].
       \]
       While the last term of the right-hand side can be estimated directly,
       we use an idea from~\cite{ST08} incorporating the stability of $I^N_p$
       and $\pi_{p+1}^N$ for the other two terms.
       \begin{align*}
        &\norm{I_p^N(\pi_{p+1}^Nw_{12}-w_{12})}{0,\Omega_{12}}^2+
        \norm{\pi_{p+1}^Nw_{12}}{0,\Omega_{12}}^2\\
         &\leq C\iint_{\Omega_{12}}(I_p^N\pi_{p+1}^Nw_{12})^2+(I_p^Nw_{12})^2+(\pi_{p+1}^Nw_{12})^2\\
         &\leq C\sum_{i=N/2+1}^N\int_{x_{i-1}}^{x_i}
                 \sum_{j=N/4+1}^{3N/4}\int_{y_{j-1}}^{y_j}(w_{12}^2(x_{i-1},y_{j-1})+w_{12}^2(x_{i-1},y_{j}))dydx\\
         &\leq C\left[
                  \sum_{i=N/2+1}^N\int_{x_{i-1}}^{x_i}e^{-\frac{2\beta x_{i-1}}{\eps}}dx
                \right]
                \left[
                  \sum_{j=N/4+1}^{3N/4}\int_{y_{j-1}}^{y_j}\bigg(e^{-\frac{2y_{j-1}}{\eps^{1/2}}}+e^{-\frac{2(1-y_{j})}{\eps^{1/2}}}\bigg)dy
                \right]\\
         &\leq C\left[
                   \int_{0}^{x_1}e^{-\frac{2\beta x_{0}}{\eps}}dx
                  +\int_{0}^{\lambda_x}e^{-\frac{2\beta x}{\eps}}dx
                \right]
                \left[
                   \int_{y_{N/4}}^{y_{N/4+1}}e^{-\frac{2\lambda_y}{\eps^{1/2}}}dy
                  +\int_{\lambda_y}^{1/2}e^{-\frac{2y}{\eps^{1/2}}}dy
                \right]\\
         &\leq C\left[
                   h+\eps
                \right]
                \left[
                   N^{-1}+\eps^{1/2}
                \right]N^{-2\sigma}
          \leq C\eps(\eps^{1/2}+N^{-1})N^{-2\sigma}.
       \end{align*}
       Here we have used the symmetry of the pointwise bound of $w_{12}$ w.r.t. $y$.
       Thus, a second bound for $(R^Nw_{12})_x$ in $\Omega_{12}$ holds:
       \begin{align*}
        \norm{(R^Nw_{12})_x}{0,\Omega_{12}}
         &\leq C \left[
                  h_{mesh}\eps^{-1}N(\ln N)^{-1/2}N^{-\sigma}(\eps^{1/4}+N^{-1/2})\eps^{1/2}+\eps^{-1/4}N^{-\sigma}
                 \right]\\
         &\leq C \eps^{-1/2}N^{-\sigma}[\eps^{1/4}+h_{mesh}(\eps^{1/4}N+N^{1/2})].
       \end{align*}
       Combining these two estimates we obtain 
       \begin{align}
        \norm{(R^Nw_{12})_x}{0,\Omega_{12}}
         &\leq C \eps^{-1/2}N^{-\sigma}\min\{
                                             (N^{-1}\max|\psi'|)^2\eps^{-3/4},h_{mesh}(\eps^{1/4}N+N^{1/2})
                                           \}\notag\\
         &\leq C\eps^{-1/2}N^{-(\sigma-1/2)}(1+h_{mesh}),
       \end{align}
       where we used $N^{-1}(\max|\psi'|)^2\leq C$ in estimating the minimum.
       The second term we want to look at is $(R^Nw_2)_x$ in $\Omega_{12}$.
       This one highlights in cancelling the logarithmic term, why $h_{mesh}$ is defined as it is.
       We obtain the two estimates
       \begin{align*}
        \norm{(R^Nw_2)_x}{0,\Omega_{12}}
         &\leq C \eps^{-1/2}N^{-\sigma}[(h+N^{-1})^2\eps^{-1/4}],\\
        \norm{(R^Nw_2)_x}{0,\Omega_{12}}
         &\leq C[h_{mesh}\eps^{-1}N(\ln N)^{-1/2}N^{-\sigma}(\eps^{1/4}+N^{-1/2})+\eps^{1/4}N^{-\sigma}]\eps^{1/2}(\ln N)^{1/2}\\
         &\leq C \eps^{-1/2}N^{-\sigma}h_{mesh}[\eps^{1/4}N+N^{1/2}]
       \end{align*}
       and therefore
       \begin{align}
        \norm{(R^Nw_2)_x}{0,\Omega_{12}}
         &\leq C\eps^{-1/2} N^{-\sigma}\min\{(h+N^{-1})^2\eps^{-1/4},h_{mesh}[\eps^{1/4}N+N^{1/2}]\}\notag\\
         &\leq C\eps^{-1/2} N^{-(\sigma-1/2)}(1+h_{mesh}).
       \end{align}
       As a third term we look at $(R^Nw_1)_x$ in $\Omega_{21}$. A similar analysis as above
       gives for $k\leq C N^{-1/4}$
       \begin{align}
        \norm{(R^Nw_1)_x}{0,\Omega_{21}}
         &\leq C\eps^{-1/2}N^{-\sigma}\min\{\eps^{-2}k^2,1+(\eps N)^{1/2}+\eps N\}\notag\\
         &\leq C\eps^{-1/2}N^{-\sigma}(1+(Nk)^{2/3})
          \leq C\eps^{-1/2}N^{-(\sigma-1/2)}
       \end{align}
       and in general ($k\leq 1$)
       \begin{align}
        \norm{(R^Nw_1)_x}{0,\Omega_{21}}
         &\leq C\eps^{-1/2}N^{-\sigma}(1+(Nk)^{2/3})
          \leq C\eps^{-1/2}N^{-(\sigma-2/3)}.\label{eq:Ru:x:proof:last}
       \end{align}
       All other terms remaining can be estimated by similar steps. Combining the results gives
       finally the statement of the theorem.
     \end{subequations}
   \end{proof}

%

  \begin{thm}\label{thm:GL:superclose}
   Let $\sigma\geq p+2$. Then it holds for the streamline-diffusion solution $u^N$
   under the restrictions on the stabilisation parameters given in Theorem~\ref{thm:Lin:superclose}
   \[
    \enorm{I^N_p u-u^N}\leq C(h+k+N^{-1}\max |\psi'|)^{p+1/2}(\max|\psi'|\ln N)^{1/2}.
   \]
  \end{thm}
  \begin{proof}
     Consider again \eqref{eq:GL:start}
     \[
      \enorm{I_p^Nu-u^N}
       \leq \enorm{\pi_p^Nu-u^N}
           +\enorm{I_p^N(\pi_{p+1}^Nu-u)-(\pi_{p+1}^Nu-u)}
           +\enorm{\pi_{p+1}^Nu-u}.
     \]
     Theorem~\ref{thm:Lin:superclose} gives under conditions on the stabilisation parameters and $\sigma\geq p+1$
     \[
      \enorm{\pi_p^Nu-u^N}\leq C(h+k+N^{-1}\max |\psi'|)^{p+1/2}(\max|\psi'|\ln N)^{1/2},
     \]
%
     Theorem~\ref{thm:Ru} gives for $\sigma\geq p+5/3$
     \[
      \enorm{I_p^N(\pi_{p+1}^Nu-u)-(\pi_{p+1}^Nu-u)}\leq C(h+k+N^{-1}\max|\psi'|)^{p+1},
     \]
%
     and Theorem~\ref{thm:interpolation} yields for $\sigma\geq p+2$
     \[
      \enorm{\pi_{p+1}^Nu-u}\leq C(h+k+N^{-1}\max |\psi'|)^{p+1}.
     \]
     Combining the three estimates completes the proof.
  \end{proof}
  \begin{rem}
   The analysis of this section shows an analytical supercloseness result of order $p+1/2$.
   In Section~\ref{sec:numerics} we will see a numerical supercloseness property of order $p+1$.
   In this sense, the analysis is not sharp.
   To our knowledge, the result of order $p+1/2$ for the SDFEM given in \cite{ST08, Fr11}
   is the best one for problems like~\eqref{eq:Lu}.
   The basic idea of the proof lies in estimating the convective term of the
   Galerkin bilinear form by the SDFEM norm. Here one order of convergence can
   be won, but half an order of the stabilisation parameter $\delta_{11}$ has
   to be spent.
   This leads to an improvement of only half an order.
  \end{rem}

\section{Superconvergence by Postprocessing}\label{sec:postprocessing}
   By utilising the supercloseness results of Theorems~\ref{thm:Lin:superclose} and~\ref{thm:GL:superclose}
   we can construct postprocessing operators. They improve our numerical solution with
   little additional computational effort to higher convergence order.

   \begin{figure}
   \begin{center}
     \begin{minipage}[c]{6cm}
      \vspace*{0cm}
      \setlength{\unitlength}{0.57pt}
      \begin{picture}(256,256)
       \multiput(  4,  0)(0,4){64}{\line(0,1){1}}
       \multiput( 12,  0)(0,4){64}{\line(0,1){1}}
       \multiput( 76,  0)(0,4){64}{\line(0,1){1}}
       \multiput(196,  0)(0,4){64}{\line(0,1){1}}

       \multiput(  0, 19)(4,0){64}{\line(1,0){1}}
       \multiput(  0, 83)(4,0){64}{\line(1,0){1}}
       \multiput(  0,173)(4,0){64}{\line(1,0){1}}
       \multiput(  0,237)(4,0){64}{\line(1,0){1}}

       \put(  0,  0){\line( 0, 1){256}}
       \put(  8,  0){\line( 0, 1){256}}
       \put( 16,  0){\line( 0, 1){256}}
       \put(136,  0){\line( 0, 1){256}}
       \put(256,  0){\line( 0, 1){256}}

       \put(  0,  0){\line( 1, 0){256}}
       \put(  0, 38){\line( 1, 0){256}}
       \put(  0,128){\line( 1, 0){256}}
       \put(  0,218){\line( 1, 0){256}}
       \put(  0,256){\line( 1, 0){256}}

       \put(-5,-5){\makebox(0,0)[t]{$0$}}
       \put(16,-5){\makebox(0,0)[t]{$\lambda_x$}}
       \put(-5,38){\makebox(0,0)[r]{$\lambda_y$}}
       \put(-5,218){\makebox(0,0)[r]{$1-\lambda_y$}}
       \put(256,-5){\makebox(0,0)[t]{$1$}}
       \put(-5,256){\makebox(0,0)[r]{$1$}}

      \end{picture}
     \end{minipage}
   \end{center}
   \caption{\label{fig:post:mesh}Macroelements $M$ of $\tilde T^{N/2}$
            constructed from $T^N$}
  \end{figure}
  Suppose $N$ is divisible by 8. We construct a coarser macro mesh $\tilde T^{N/2}$
  composed of macro rectangles $M$, each consisting of four rectangles
  of~$T^N$.
  The construction of these macro elements $M$ is done such that the union on them covers
  $\Omega$ and none of them crosses the transition lines at $x=\lambda_x$ and at $y=\lambda_y$ or $y=1-\lambda_y$,
  see Figure~\ref{fig:post:mesh}.
  Remark that in general $\tilde T^{N/2}\neq T^{N/2}$ due to different transition
  points $\lambda_x$ and $\lambda_y$ and the mesh generating function $\phi$.

  We now define postprocessing operators locally for one macro element $M\in\tilde T^{N/2}$.
  The first one was presented in 1d in \cite{Tob06} (for $p\geq 3$), a modification of an operator given in \cite{Lin91}.

  Let $\hat v$ be the linearly mapped function $v$ from any interval $[x_{i-1},x_{i+1}]$ onto the reference interval $[-1,1]$.
  Note that $x_i$ is not necessarily mapped onto $0$, but to a value $a\in(-1,1)$.
  In \cite{Franz08} the following condition on the underlying mesh is given, that guarantees
  the non-degenerate behaviour of the macro elements and of the operators defined on it:
  There exists a constant $q\geq1$ independent of $N$ and $\eps$ such that
   \begin{align*}
    \frac{\max\{h_i,h_{i+1}\}}{\min\{h_i,h_{i+1}\}}&\leq q,\quad
       \mbox{for all }i=1,\dots,N/2-1,\\
    \frac{\max\{k_j,k_{j+1}\}}{\min\{k_j,k_{j+1}\}}&\leq q,\quad
       \mbox{for all }j=1,\dots,N/4-1\mbox{ and }j=3N/4+1,\dots,N-1.
   \end{align*}
  A Shishkin mesh has $q=1$ while a Bakhvalov--Shishkin mesh has
  $q=\ln(3)/\ln(5/3)$. For many more S-type meshes this condition holds
  and we assume it further-on.
  
  Define the postprocessing operator\footnote{changed in March 2023 for $p\geq 2$, numerical experiments used already this definition} 
  $\widehat{P}_{vec}:C[-1,1]\to \PS_{p+1}[-1,1]$ by
  \begin{align*}
    \mbox{for $p=1$:}\hspace*{1cm}
    \widehat{P}_{vec} \hat v(-1)&=v(x_{i-1}),\qquad
    \widehat{P}_{vec} \hat v( a) =v(x_{i}),\qquad
    \widehat{P}_{vec} \hat v( 1) =v(x_{i+1}),\\
    \mbox{for $p=2$:}\hspace*{1cm}
    \widehat{P}_{vec} \hat v(-1)&=v(x_{i-1}),\qquad
    \widehat{P}_{vec} \hat v( 1) =v(x_{i+1}),\\
    \int_{-1}^{a} (\widehat{P}_{vec}\hat v-\hat v)&=0,\quad
    \int_{a}^1    (\widehat{P}_{vec}\hat v-\hat v) =0,\\
    \mbox{while for $p\geq 3$:}\hspace*{1cm}
    \widehat{P}_{vec} \hat v(-1)&=v(x_{i-1}),\qquad
    \widehat{P}_{vec} \hat v( a) =v(x_{i}),\qquad
    \widehat{P}_{vec} \hat v( 1) =v(x_{i+1}),\\
    \int_{-1}^{a} (\widehat{P}_{vec}\hat v-\hat v)&=0,\quad
    \int_{a}^1    (\widehat{P}_{vec}\hat v-\hat v) =0,\\
    \int_{-1}^1   (\widehat{P}_{vec}\hat v-\hat v)p&=0,\quad p\in \PS_{p-3}[-1,1]\setminus\R.
  \end{align*}

  By using the tensor product structure we obtain the full postprocessing operator 
  $P_{vec,M}:C(M)\to \QS_{p+1}(M)$ on each macro element.
  Then, this piecewise projection is extended to a global, continuous
  function by setting
  \begin{gather*}
     \bigl(P^{p+1}_{vec}v\bigr)(x,y) := \bigl(P_{vec,M} v\bigr)(x,y)
                \quad \text{for} \ (x,y)\in M.
  \end{gather*}

  The second postprocessing operator is defined by using only point evaluations.
  Let $M$ be the union of mesh cells $M_1$, $M_2$, $M_3$ and $M_4$. In each of them
  we sample the Gau\ss-Lobatto points. Due to the tensor structure
  we can order the $x$- and $y$-coordinates of those points and obtain the sequences $\{\tilde x_i\}$
  and $\{\tilde y_j\}$, $i,j=0,\dots,2p$.
  
  Let $P_{GL,M}:C(M)\to \QS_{p+1}(M)$ denote the projection/interpolation operator
  fulfilling
  \[
   P_{GL,M}v(\tilde x_i,\tilde y_j)=v(\tilde x_i,\tilde y_j),\quad i,j=0,\,1,\,3,\,5,\dots,2p-3,2p-1,\,2p.
  \]
  Then, this piecewise projection is extended to a global, continuous
  function by setting
  \begin{gather*}
     \bigl(P^{p+1}_{GL}v\bigr)(x,y) := \bigl(P_{GL,M} v\bigr)(x,y)
                \quad \text{for} \ (x,y)\in M.
  \end{gather*}

  \begin{lem}\label{lem:post}
   For the postprocessing operators $P^{p+1}_{GL}$ and $P^{p+1}_{vec}$ defined above holds
  \begin{subequations}
   \begin{align}
    P^{p+1}_{GL}I_p^Nv    &= P^{p+1}_{GL}v,&
    P^{p+1}_{vec}\pi_p^Nv &= P^{p+1}_{vec}v,&
    \mbox{for all }v\in C(\Omega),\label{eq:post:1}\\
    \enorm{P^{p+1}_{GL}v^N}  &\leq C\enorm{v^N},&
    \enorm{P^{p+1}_{vec}v^N} &\leq C\enorm{v^N},&
    \mbox{for all }v^N\in V^N\label{eq:post:2}.
   \end{align}
   Let $u$ be the solution of \eqref{eq:Lu}, Assumption~\ref{ass:dec} be true and $\sigma\geq p+2$.
   Then it holds
   \begin{align}
    \enorm{P^{p+1}_{GL}u-u}+
    \enorm{P^{p+1}_{vec}u-u}\leq C(h+k+N^{-1}\max|\psi'|)^{p+1}\label{eq:post:3}.
   \end{align}
   \end{subequations}
  \end{lem}
  \begin{proof}
   The consistency \eqref{eq:post:1} is a direct consequence of the definitions of $P^{p+1}_{GL}$ and $P^{p+1}_{vec}$.
   The stability \eqref{eq:post:2} can be shown for both operators similarly. Therefore, let $P_M$
   be any of the local operators $P_{GL,M}$ and $P_{vec,M}$.
   The stability \eqref{eq:post:2} then follows by showing
   $\norm{P_Mv^N}{0,M}\leq C\norm{v^N}{0,M}$ and
   $\snorm{P_Mv^N}{1,M}\leq C\snorm{v^N}{1,M}$ for any $M\in\tilde T_{N/2}$ and any $v^N\in V^N(M)$.
   The operator $P_M$ is a linear operator from the finite dimensional space $V^N(M)$ into the finite
   dimensional space $\QS_{p+1}(M)$. Thus it is continuous and with $v^N\mapsto \norm{v^N}{0,M}$ being a
   norm in both spaces we obtain
   \[
    \norm{P_M v^N}{0,M}\leq C \norm{v^N}{0,M},\quad \mbox{for all }v^N\in V^N(M).
   \]
   Similarly $v^N\mapsto\snorm{v^N}{1,M}$ is a norm on the quotient space $V^N(M)\setminus\R$
   and $v^N\mapsto\norm{v^N}{1,M}$ is a norm on $\QS_{p+1}(M)$. Therefore,
   \[
    \snorm{P_M v^N}{1,M}
    \leq\norm{P_M v^N}{1,M}
    \leq C\snorm{v^N}{1,M}\quad \mbox{for all }v^N\in V^N(M)\setminus\R.
   \]
   Finally, the interpolation error \eqref{eq:post:3} follows by Assumption~\ref{ass:dec}
   and standard anisotropic estimates for interpolation~\cite{Apel99,FrM10_1}.
  \end{proof}

  \begin{thm}\label{thm:post}
   Let $\sigma\geq p+2$. Then it holds for the streamline-diffusion solution $u^N$
   under the restrictions on the stabilisation parameters given in Theorem~\ref{thm:Lin:superclose}
   \[
     \enorm{u-P^{p+1}_{GL}u^N}
    +\enorm{u-P^{p+1}_{vec}u^N}\leq C(h+k+N^{-1}\max|\psi'|)^{p+1/2}(\max|\psi'|\ln N)^{1/2}.
   \]
  \end{thm}
  \begin{proof}
   Using the consistency and stability of $P^{p+1}_{GL}$ we obtain
   \begin{align*}
    \enorm{u-P^{p+1}_{GL}u^N}
    &\leq\enorm{u-P^{p+1}_{GL}u}+\enorm{P^{p+1}_{GL}I^N_p u-P^{p+1}_{GL}u^N}\\
    &\leq\enorm{u-P^{p+1}_{GL}u}+C\enorm{I^N_p u-u^N}.
   \end{align*}
   Similarly one can show
   \[
    \enorm{u-P^{p+1}_{vec}u^N}
     \leq\enorm{u-P^{p+1}_{vec}u}+C\enorm{\pi^N_p u-u^N}.
   \]
   Now the statement follows by \eqref{eq:post:3} and the supercloseness results of Theorems~\ref{thm:Lin:superclose}
   and~\ref{thm:GL:superclose}.
  \end{proof}

\section{Numerical Example}\label{sec:numerics}

   Let us consider the singularly perturbed problem given by
   \begin{align*}
      -\eps \laplace u - (2-x) u_x + 3/2 u &= f,\quad\mbox{in }\Omega=(0,1)^2\\
      u&=0,\quad\mbox{on }\partial\Omega,
   \end{align*}
   with a constructed right-hand side, such that
   \[
      u(x,y) = \left(\cos(x\pi/2)-\frac{e^{-x/\eps}-e^{-1/\eps}}{1-e^{-1/\eps}}\right)
               \frac{(1-e^{-y/\eps^{1/2}})(1-e^{-(1-y)/\eps^{1/2})}}{1-e^{-1/\eps^{1/2}}}
   \]
   is the exact solution.
   
   The following calculations were done in Matlab and the linear systems solved by its
   direct ``backslash-solver''. We fix  the polynomial degree to $p=3$ and the parameter
   for the Shishkin mesh to $\sigma=p+2=5$. Moreover, we set $\eps=10^{-6}$, sufficiently small
   to generate the sharp boundary layers we are interested in. Note that additional computations
   were done with different polynomial degrees $p$ and varied perturbation parameters $\eps$
   supporting the same conclusions.
  \begin{table}[tbp]
   \begin{center}
    \caption{Convergence and closeness errors for $Q_3$-elements on a Shishkin mesh for $\eps=10^{-6}$
             with corresponding rates $r_N^S$, expected rates are 3 and 3.5
             \label{tab:conv:S}}
    \begin{tabular}{r cr cr cr cr}
           \multicolumn{1}{r}{N}\rule{0pt}{1.1em}
         & \multicolumn{2}{c}{$\enorm{u-u^N}$}
         & \multicolumn{2}{c}{$\enorm{\pi_3^N u-u^N}$}
         & \multicolumn{2}{c}{$\enorm{I_3^Nu-u^N}$}
         & \multicolumn{2}{c}{$\enorm{J_3^Nu-u^N}$}\\
     \hline\rule{0pt}{1.1em}
      8 & 2.270e-02 & 2.60 & 8.259e-03 & 2.55 & 9.587e-03 & 2.72 & 9.811e-03 & 2.68\\
     16 & 7.926e-03 & 2.79 & 2.940e-03 & 2.71 & 3.184e-03 & 2.81 & 3.309e-03 & 2.79\\
     32 & 2.141e-03 & 2.96 & 8.233e-04 & 3.28 & 8.515e-04 & 3.31 & 8.923e-04 & 3.24\\
     64 & 4.723e-04 & 3.04 & 1.544e-04 & 4.25 & 1.572e-04 & 4.24 & 1.707e-04 & 3.83\\
    128 & 9.156e-05 & 3.01 & 1.563e-05 & 4.53 & 1.600e-05 & 4.50 & 2.169e-05 & 3.46\\
    256 & 1.699e-05 &      & 1.240e-06 &      & 1.292e-06 &      & 3.122e-06 &
    \end{tabular}
   \end{center}
  \end{table}
  \begin{table}[tbp]
   \begin{center}
    \caption{Convergence and closeness errors for $Q_3$-elements on a Bakhvalov--Shishkin mesh for $\eps=10^{-6}$
             with corresponding rates $r_N^B$, expected rates are 3 and 3.5
             \label{tab:conv:BS}}
    \begin{tabular}{r cr cr cr cr}
           \multicolumn{1}{r}{N}\rule{0pt}{1.1em}
         & \multicolumn{2}{c}{$\enorm{u-u^N}$}
         & \multicolumn{2}{c}{$\enorm{\pi_3^N u-u^N}$}
         & \multicolumn{2}{c}{$\enorm{I_3^Nu-u^N}$}
         & \multicolumn{2}{c}{$\enorm{J_3^Nu-u^N}$}\\
     \hline\rule{0pt}{1.1em}
      8 & 4.104e-03 & 2.74 & 1.673e-03 & 3.07 & 1.774e-03 & 3.11 & 1.828e-03 & 3.05\\
     16 & 6.145e-04 & 2.92 & 1.986e-04 & 4.22 & 2.055e-04 & 4.18 & 2.207e-04 & 3.67\\
     32 & 8.106e-05 & 2.94 & 1.065e-05 & 4.33 & 1.137e-05 & 4.25 & 1.729e-05 & 3.21\\
     64 & 1.055e-05 & 2.97 & 5.309e-07 & 4.17 & 5.960e-07 & 4.11 & 1.868e-06 & 3.02\\
    128 & 1.349e-06 & 2.98 & 2.950e-08 & 4.08 & 3.440e-08 & 4.05 & 2.310e-07 & 2.99\\
    256 & 1.707e-07 &      & 1.742e-09 &      & 2.077e-09 &      & 2.902e-08 &
        \end{tabular}
   \end{center}
  \end{table}
   
   In the following tables, the experimental rates of of convergence for given measured errors $e_N$
   are calculated by
   \[
    p_N^S=\frac{\ln(e_N/e_{2N})}{\ln(2\ln(N)/\ln(2N))},\quad
    p_N^B=\frac{\ln(e_N/e_{2N})}{\ln(2)},
   \]
   assuming $e_N=C(N^{-1}\ln N)^{p_N^S}$ on a Shishkin mesh and $e_N=CN^{-p_N^B}$ on a Bakhvalov--Shishkin mesh.
   
   The numerical method is the SDFEM given in \eqref{eq:SD_form} with stabilisation 
   parameters according to the upper bounds in Theorem~\ref{thm:Lin:superclose}, where $C$ is set to $1$.
   
   Tables~\ref{tab:conv:S} and \ref{tab:conv:BS}
   present the convergence and closeness results. We observe third order convergence of the
   numerical method as predicted by Corollary~\ref{cor:convergence}.
   Moreover, the results on a graded mesh like the Bakhvalov--Shishkin mesh are much better
   compared with the piecewise equidistant Shishkin mesh. There are two orders of magnitude
   difference in the final line for the energy error.
   
   The closeness results for the vertex-edge-cell interpolant $\pi_3^N$ and
   the Gau\ss-Lobatto interpolant $I_3^N$ as well as a pointwise interpolation operator 
   $J_3^N$ using equidistantly spaced interpolation points are also given. Supercloseness of 
   a better order than $p+1/2=3.5$ can clearly be seen for $\pi_3^N$ and $I_3^N$, whereas
   for $J_3^N$ the rate is not as high. On the Bakhvalov--Shishkin mesh, we observe a clear order 4
   for the first two operators and an order 3 for the equidistant interpolant.

   Table~\ref{tab:post:S} 
  \begin{table}[tbp]
   \begin{center}
    \caption{Superconvergence errors for $Q_3$-elements on a Shishkin mesh 
             and a Bakhvalov--Shishkin mesh with corresponding rates $r_N^S$ and $r_N^B$, resp., for $\eps=10^{-6}$, expected rates are 3.5
             \label{tab:post:S}}
    \begin{tabular}{r cr cr| cr cr}
         & \multicolumn{4}{c|}{Shishkin} &
           \multicolumn{4}{c}{Bakhvalov-Shishkin}\\
           \multicolumn{1}{r}{N}\rule{0pt}{1.1em}
         & \multicolumn{2}{c}{$\enorm{u-P_{vec}^Nu^N}$}
         & \multicolumn{2}{c|}{$\enorm{u-P_{GL}^Nu^N}$}
         & \multicolumn{2}{c}{$\enorm{u-P_{vec}^Nu^N}$}
         & \multicolumn{2}{c}{$\enorm{u-P_{GL}^Nu^N}$}\\
     \hline\rule{0pt}{1.1em}
      8 & 2.940e-02 & 2.79 & 2.883e-02 & 2.77 & 1.071e-02 & 3.58 & 1.282e-02 & 4.10\\
     16 & 9.471e-03 & 3.16 & 9.362e-03 & 3.21 & 8.965e-04 & 3.99 & 9.771e-04 & 3.71\\
     32 & 2.142e-03 & 3.54 & 2.069e-03 & 3.59 & 5.633e-05 & 4.07 & 6.141e-05 & 3.99\\
     64 & 3.511e-04 & 4.03 & 3.298e-04 & 3.96 & 3.343e-06 & 4.06 & 3.587e-06 & 4.09\\
    128 & 4.001e-05 & 4.11 & 3.892e-05 & 4.06 & 2.003e-07 & 4.04 & 2.110e-07 & 4.06\\
    256 & 4.005e-06 &      & 4.005e-06 &      & 1.217e-08 &      & 1.266e-08 &   
    \end{tabular}
   \end{center}
  \end{table}   
   %
   %
   now shows the results of the postprocessed numerical solutions. We observe for both postprocessing operators
   on both meshes convergence rates of order $p+1=4$, which compared with the convergence results presented in
   Tables~\ref{tab:conv:S} and \ref{tab:conv:BS}
   is an increase of a full order. Theorem~\ref{thm:post} only predicted half an order increase. 
   Thus, it seems that the supercloseness result of Theorem~\ref{thm:Lin:superclose}
   is not sharp. By improving this estimate, the improvement of Theorem~\ref{thm:post} would follow immediately.

   By comparing the results of the two operators, we observe very little difference for larger values of $N$.
   Thus the new postprocessing operator using only point values has comparable convergence properties to the
   already existing one that uses integral values.\\

   \textbf{Acknowledgement}: The author would like to thank Lars Ludwig for carefully reading the manuscript
   and suggesting improvements to the readability of the paper.
   
  \bibliographystyle{plain}
  \bibliography{superGL}

\end{document}